%% file: LlosaIsenrich-From-the-second-BNSR-to-Dehn-functions-of-coabelian-subgroups.tex
\documentclass[10pt,a4paper,reqno]{amsart}

\usepackage{amsmath}
\usepackage{amsfonts}
\usepackage{amssymb}
\usepackage{amsthm}
\usepackage{mathrsfs}
\usepackage[shortlabels]{enumitem}
\usepackage{graphicx}
\usepackage{color}
\usepackage{dsfont}
\usepackage[latin1]{inputenc}
\usepackage{MnSymbol}
\usepackage{enumitem}
\usepackage{extpfeil}
\usepackage{mathtools}
\usepackage{url}
\usepackage[margin=1.25in]{geometry}
\usepackage{fancyhdr}
\usepackage[all,cmtip]{xy}
\usepackage{hyperref}
\usepackage{tikz}
\usepackage{thm-restate}

\numberwithin{equation}{section}

\newcommand{\R}{{\mathds R}}

\newcommand{\Z}{{\mathds Z}}

\def\ker{{\rm{ker}}}

\theoremstyle{plain}
\newtheorem{theorem}{Theorem}[section]

\newtheorem{corollary}[theorem]{Corollary}

\newtheorem{lemma}[theorem]{Lemma}

\newtheorem*{question*}{Question}

\theoremstyle{definition}
\newtheorem{remark}[theorem]{Remark}
\newtheorem*{acknowledgements*}{Acknowledgements}

\newtheorem{definition}[theorem]{Definition}
\newtheorem*{notation*}{Notation}
\newtheorem*{convention*}{Convention}

\title{From the second BNSR invariant to Dehn functions of coabelian subgroups}
\author{Claudio Llosa Isenrich}
\address{Faculty of Mathematics, Karlsruhe Institute of Technology, Englerstr. 2, 76131 Karlsruhe, Germany}
\email{claudio.llosa@kit.edu}

\thanks{}

\begin{document}

\begin{abstract}
Given a finitely presented group $G$ and a surjective homomorphism $G\to \mathbb{Z}^n$ with finitely presented kernel $K$, we give an upper bound on the Dehn function of $K$ in terms of an area-radius pair for $G$. As a consequence we obtain that finitely presented coabelian subgroups of hyperbolic groups have polynomially bounded Dehn function. This generalises results of Gersten and Short and our proof can be viewed as a quantified version of results from Renz' thesis on the second BNSR invariant.
\end{abstract}

\maketitle

\section{Introduction}

A natural problem in geometric group theory that has recently received a lot of attention is understanding the geometry of subgroups of a given group $G$. One important source of interesting subgroups are kernels of surjective homomorphisms $G\to \mathbb{Z}^n$; we will call such subgroups \emph{coabelian} subgroups of $G$. This is because methods based on BNSR invariants \cite{BNS-87,Ren-88,BieRen-88} and Morse theory \cite{BesBra-97} allow us to understand the finiteness properties of such kernels, which has enabled the construction of finitely presented subgroups with interesting geometric properties.

A classical geometric invariant of a finitely presented group $G=\left\langle X\mid R\right\rangle$ is its Dehn function $\delta_G(n)$, which is defined as the maximal number of conjugates of relations required to detect if a word of length $\leq n$ in $X$ represents the trivial element of $G$. It can also be interpreted geometrically as an optimal isoperimetric function for loops in the universal cover of a closed manifold with fundamental group $G$. 

Hyperbolic groups form an important class of groups in geometric group theory. One way to define them is as the class of all groups with linear Dehn function \cite{Gro-87} (see also \cite[Section III.H.2]{BriHae-99} and the references therein). This naturally leads to the question which functions can arise as Dehn functions of finitely presented subgroups of hyperbolic groups. It was first raised by Brady \cite{Bra-99}, following his construction of the first example of a finitely presented non-hyperbolic subgroup of a hyperbolic group, and has also been raised by Brady and Tran \cite{BraTra-21} in the context of the recent constructions of new examples of finitely presented non-hyperbolic subgroups of hyperbolic groups \cite{Lod-18,Kro-21,IMM-23,LIMP-21,LloPy-22, KroLlo-23, LloPy-23}. By a result of Gersten and Short \cite{GerSho-02} the Dehn functions of kernels of homomorphisms from hyperbolic groups onto $\mathbb{Z}$ are polynomially bounded, putting strong constraints on their nature. However, this result can not be applied to the most recent examples of non-hyperbolic finitely presented subgroups of hyperbolic groups constructed by Kropholler and the author \cite{KroLlo-23}, and by Py and the author \cite{LloPy-23}, since they are kernels of surjective homomorphisms onto $\mathbb{Z}^n$ with $n>1$. This raises the question if their Dehn functions could behave differently. The main motivation for this work is to show that they are also polynomially bounded. In particular, this shows that if a finitely presented subgroup of a hyperbolic group with non-polynomially bounded Dehn function exists, then it has to be found by different methods.

\begin{theorem}\label{thm:Dehn-hyp}
    Let $G$ be a hyperbolic group and let $\phi:G\to Q$ be a homomorphism to a finitely generated abelian group with finitely presented kernel $N:=\ker(\phi)$. 
    Then there is a polynomial $p(n)$ such that $\delta_N(n)\leq p(n)$.
\end{theorem}

While a priori it might be possible to use the methods by Gersten and Short \cite{GerSho-02}, as well as the methods in this work to obtain explicit numerical upper bounds on the degree of the polynomial $p(n)$, in practice this turns out to be difficult for any concrete example. This is because attaining such a bound would require finding suitable explicit finite presentations for which our proof can be quantified. In forthcoming work \cite{BraKroLloSor-24} Brady, Kropholler, Soroko and the author  will use techniques based on pushing fillings in cube complexes to obtain the first explicit numerical upper bounds on the polynomial degrees of the Dehn functions for some of the examples constructed by Brady, Lodha and Kropholler \cite{Bra-99, Lod-18, Kro-21}.

Theorem \ref{thm:Dehn-hyp} will be a consequence of the following generalisation of Gersten and Short's Theorem B in \cite{GerSho-02} to arbitrary finitely generated abelian quotient groups.

\begin{theorem}\label{thm:Main-Dehn}
    Let $G$ be a finitely presented group that fits into an extension
    \[
        1\to N \to G \to Q\to 1
    \]
    with $Q$ finitely generated abelian and $N$ finitely presented. Assume that $(f,g)$ is an area-radius pair for $G$ with $f(n)\geq n$. Then there is a constant $A>1$ such that the Dehn function $\delta_N(n)$ of $N$ is bounded above by $n\mapsto A^{g(n)} f(n)$.
\end{theorem}

We recall that a pair $(f,g)$ is an area-radius pair for a finite presentation of $G$ if for every word of length at most $n$ there is a van Kampen diagram with area $\leq f(n)$ and radius $\leq g(n)$. We refer to Section \ref{sec:Dehnfunctions} for further details on Dehn functions and area-radius pairs and to \cite{Bri-02} for the definition of a van Kampen diagram.

To prove Theorem \ref{thm:Main-Dehn} we first carefully review Renz' proof of one of his main thesis results \cite[Satz C]{Ren-88}, which relates subspheres of the second homotopical BNSR invariant of a group to finite presentability of kernels of certain morphisms to free abelian groups (see Theorem \ref{thm:homotopical-BNSR}). Renz' proof involves the construction of a van Kampen diagram for every loop in a $K(N,1)$. We will then produce an upper bounds on the Dehn function of $N$ by proving that one can bound the area of the van Kampen diagrams obtained from this construction from above in terms of an area-radius pair for the group $G$.

Besides the application of Theorem \ref{thm:Main-Dehn} to subgroups of hyperbolic groups, we can also deduce the following generalisation of \cite[Theorem A]{GerSho-02} for automatic groups.

\begin{corollary}\label{cor:Dehn-automatic}
    Let $G$ be a finitely presented (synchronously or asynchronously) automatic group and let $\phi:G\to Q$ be a homomorphism to a finitely generated abelian group with finitely presented kernel $N:=\ker(\phi)$.
    Then there is a $C>1$ such that $\delta_N(n)\leq C^n$.
\end{corollary}
The upper bound in Corollary \ref{cor:Dehn-automatic} is optimal: hyperbolic groups are automatic and Bridson gave an example of a cocyclic subgroup of a direct product of two hyperbolic groups with exponential Dehn function \cite[Section 2.5]{Bri-01}. Since direct products of finitely many finitely generated free groups are automatic, Corollary \ref{cor:Dehn-automatic} has the following interesting consequence:

\begin{corollary}\label{cor:Dehn-SFP}
    If $N$ is a finitely presented coabelian subgroup of a direct product of finitely many finitely generated free groups, then there is a $C>1$ such that $\delta_N(n)\leq C^n$.
\end{corollary}
 This is related to a conjecture by Bridson (it was also raised as a question by Dison in his PhD thesis \cite{Dis-08}), which asserts that all finitely presented subgroups of a direct product of finitely many finitely generated free groups have polynomially bounded Dehn function. While progress on Bridson's conjecture has been made for specific classes of subgroups of direct products of free groups, it remains open in general, even for coabelian subgroups. In particular, to our knowledge Corollary \ref{cor:Dehn-SFP} has not been stated explicitly anywhere, although it can probably also be deduced using the methods developed in \cite{Dis-08}. We refer to \cite{KroLlo-23-II, LloTes-20} for further context and results on this problem. In forthcoming work with Ascari, Bertolotti, Italiano and Migliorini \cite{ABILIM-24} we will give a positive answer to a uniform version of Bridson's conjecture raised by Dison \cite[Question 1]{Dis-08} for all subgroups of a direct product of $n$ non-abelian free groups of finiteness type $F_{n-1}$ and not $F_n$; this will solve this question completely in the case of three factors.\vspace{.2cm}

 {\noindent \bf Structure.} In Section \ref{sec:BNSR-invariants} we introduce homological and homotopical BNSR invariants and discuss some of their most important properties. In Section \ref{sec:Renz-proof} we give an account of Renz' proof of Theorem C from his thesis \cite{Ren-88} for $n=2$. Our discussion in Sections \ref{sec:BNSR-invariants} and \ref{sec:Renz-proof} closely follows Renz' presentation in \cite{Ren-88} and its purpose is to introduce his construction of van Kampen diagrams whose area we will later bound in Section \ref{sec:Dehnfunctions}. While making some small changes to Renz' presentation to suit our intended use, we do not claim to prove new results in those sections. However, to our knowledge this is the first English account of this part of Renz' proof and may thus be of independent interest. In Section \ref{sec:Dehnfunctions} we prove upper bounds on the area of the van Kampen diagrams constructed in Section \ref{sec:Renz-proof} and then apply these bounds to show Theorem \ref{thm:Main-Dehn}. \vspace{.2cm}

 {\noindent \bf Acknowledgements.} We thank Martin Bridson for helpful comments and for pointing out the optimality of Corollary \ref{cor:Dehn-automatic}.

\section{BNSR invariants and finiteness properties}\label{sec:BNSR-invariants}

Introduced by Bieri, Neumann, Strebel and Renz, BNSR invariants characterise the finiteness properties of kernels of homomorphisms onto free abelian groups. We say that a group $G$ is of finiteness type $F_n$ if there is a classifying CW complex $K(G,1)$ with finitely many cells of dimension $\leq n$. We say that a group is of type $F_{\infty}$ if it is of type $F_n$ for all $n$ and of type $F$ if it has a finite $K(G,1)$. Finiteness properties naturally generalise finite generation ($\Leftrightarrow$ $F_1$) and finite presentability ($\Leftrightarrow$ $F_2$). One can view the finiteness property $F_n$ as a homotopical finiteness property of a group and can also define more algebraic homological finiteness properties over unitary rings $R$. 
 
 For an integer $n\geq 0$, we say that a group $G$ is of type $FP_n(R)$ if there is a partial projective resolution
 \[
    P_n\to P_{n-1}\to \cdots \to P_0\to R\to 0
 \]
 of the trivial $RG$-module $R$ by finitely generated projective $RG$-modules $P_i$ up to dimension $n$. If $G$ is of type $F_n$, the $G$-action on the universal cover of a classifying space for $G$ provides such a resolution, showing that $F_n$ implies $FP_n(R)$ for all rings $R$. One can also show that $FP_1(R)$ and $F_1$ are equivalent. Examples of Bestvina and Brady show that for $n\geq 2$ the property $FP_n(R)$ is strictly weaker than $F_n$ \cite{BesBra-97}. However, if $G$ is of type $F_2$, then $FP_n(\Z)$ implies $F_n$ (see \cite[VIII.7]{Bro-82}).

 The character sphere $S(G)$ of a finitely generated group $G$ is the quotient space
 \[
    S(G):=\left({\rm Hom}(G,\R)\setminus \left\{0\right\}\right)/\sim,
 \]
 where we call two characters $\chi_1,\chi_2:G\to \mathbb{R}$ equivalent and write $\chi_1\sim \chi_2$ if there is a positive real number $\lambda>0$ such that $\chi_1=\lambda\chi_2$. We will denote by $\left[\chi\right]$ the equivalence class of $\chi$ in $S(G)$.

In his thesis \cite{Ren-88} Renz introduced a sequence of invariants for a group $G$ of type $F_n$ which arise as subsets of the character sphere and describe the homotopical finiteness properties of kernels of elements of $S(G)$:
\[
    {}^{\ast}\Sigma^{n}(G)\subseteq {}^{\ast}\Sigma^{n-1}(G)\subseteq \cdots \subseteq {}^{\ast}\Sigma^{0}(G)=S(G).
\]

These generalise a set of invariants $\Sigma^{n}(G,R)$ introduced by Bieri, Neumann and Strebel for $n=1$ and by Bieri and Renz for $n\geq 2$ which characterise the homological finiteness properties $FP_n(R)$ over unitary rings $R$. 

To define the invariants ${}^{\ast}\Sigma^k(G)$ for a group $G$ of type $F_k$, we first need to define valuations on free $G$-complexes. We will then use them to define a valuation on the universal covering of a $K(G,1)$ for $G$. The definition of a valuation on a free $G$-complex is similar to the definition of a valuation on a free $\Z G$-module.

\begin{definition}
Let $C$ be a free $G$-complex and let $\chi\in {\rm Hom}(G,\R)$. A continuous map $v_{\chi}:C\to \R$ is called a \emph{valuation} on $C$ associated with $\chi$, if the following properties hold:
\begin{enumerate}
    \item $v_{\chi}(gc)=\chi(g)+v_{\chi}(c)$ for all $c\in C$ and $g\in G$.
    \item $v_{\chi}(C^0)\subseteq \chi(G)$
    \item If $\sigma$ is a cell of $C$ with boundary $\partial \sigma$, then $\min(v_{\chi}(\partial\sigma))\leq v_{\chi}(c)\leq \max(v_{\chi}(\partial \sigma))$ for all $c\in \sigma$.
\end{enumerate}
To simplify notation we will often write $v$ instead of $v_{\chi}$. Moreover, for a cell $\sigma$ of $C$ we will write $v(\sigma):=\min\left\{v(c)\mid c\in \sigma\right\}$ to denote the minimum of $v$ on $\sigma$. 
\end{definition}

\begin{remark}\label{rmk:discretized-valuation}
    By assumption (3) the minimum $v(\sigma):=\min \left\{v(c)\mid c\in \sigma\right\}$ will necessarily be attained in a 0-cell of $C$. We can thus view the induced map on cells of $C$ as a discretized version of $v$ which takes values in $\chi(G)$. In Section \ref{sec:Renz-proof} we will work with this discretized valuation.
\end{remark}

From now on we will assume that we have chosen a $K(G,1)$ with a single $0$-cell $\sigma^0$.  We recall that for a given presentation $\left\langle X\mid R\right\rangle$ of $G$ one can always construct such a $K(G,1)$, which we will denote by $K$ by starting from a $0$-cell, then attaching 1-cells for all generators and then building the 2-skeleton by attaching $n$-gons for all (reduced) relations of length $n$ along the words described by these relations. Finally, we construct $K$ by inductively attaching $n$-cells along generating sets to kill all higher homotopy groups. Note that the 2-skeleton $K^2$ of this complex is the \emph{Cayley complex} of $G$ with respect to its presentation $\left\langle X\mid R\right\rangle$.

Given a character $\chi\in {\rm Hom}(G,\R)\setminus \left\{0\right\}$ we can always define a valuation $v_{\chi}$ inductively on the skeleta of the universal cover $C$ of $K$ as follows \cite{Ren-88}. We fix a lift $\widetilde{\sigma}^0$ of $\sigma^0$ and define $v_{\chi}(g\cdot \widetilde{\sigma}^0):= \chi(g)$. We then choose a simplicial subdivision of $K$ which restricts to simplicial subdivisions of the cells. This induces a $G$-invariant simplicial subdivision of $C$. We then define $v_{\chi}$ inductively on $C$ so that it is affine linear on simplices as follows. Assume that we defined $v_{\chi}$ on the restriction of the simplicial structure to the $i$-skeleton $C^i$ of $C$. We extend $v_{\chi}$ to $(i+1)$-cells by first choosing auxiliary values on the $0$-simplices which lie in the interior of a $(i+1)$-cell $\sigma$ of $C$ in a $G$-equivariant way (take e.g. the average of $v_{\chi}$ over all $0$-cells of $\partial\sigma$), and then extending affine linearly on simplices. It is clear that $v_{\chi}$ will then define a valuation on $C$. 

We now assume that we are given a free $G$-complex $C$ with a single $G$-orbit of $0$-cells, a character $\chi\in {\rm Hom}(G,\R)\setminus\left\{0\right\}$ and a valuation $v$ on $C$ associated with $\chi$ satisfying the above assumptions. For $r\in \R$, let $C_{v,r}$ be the full subcomplex of $C$ spanned by the subset $\left\{g\widetilde{\sigma}^0 \mid v(g\widetilde{\sigma}^0)\geq - r\right\}\subseteq C^0$. Moreover, denote $C_v:= C_{v,0}$ and call $C_v$ the valuation subcomplex associated to $v$.\footnote{We will only use the notations $C_v$ and $C_{v,r}$ in this section and Section \ref{sec:Renz-proof-1}, while in Sections \ref{sec:Renz-proof-2} and \ref{sec:Dehnfunctions} we will use the notation $C_r$ with $r\in \R$ to denote a different subset of the universal cover of a presentation 2-complex for $G$.} We can now define the BNSR invariants of $G$.

\begin{definition}
    Let $G$ be a group of type $F_k$ and let $\left[\chi\right]\in S(G)$. Then $\left[\chi\right]\in {}^{\ast}\Sigma^k(G)$ if and only if there is a classifying space $K=K(G,1)$ for $G$ with finite $k$-skeleton and a valuation $v_{\chi}$ on the universal covering $C$ of $K$ such that the valuation subcomplex $C_v$ is $(k-1)$-connected.
\end{definition}

According to \cite{Ren-88} if $\left[\chi\right]\in {}^{\ast}\Sigma^k(G)$, this does not imply that for every CW complex $K$ with finite $k$-skeleton and every valuation $v_{\chi}$ we have that $C_v$ is $(k-1)$-connected. However, there is a weaker criterion which is independent of all choices involved. 
\begin{definition}
    We say that the valuation subcomplex $C_v$ of $C$ is \emph{essentially $(k-1)$-connected} if there is a real number $r>0$ such that the inclusion $\iota:C_v\hookrightarrow C_{v,r}$ induces trivial homomorphisms $\iota_{\ast}:\pi_i(C_{v})\to \pi_i(C_{v,r})$ for $0\leq i \leq k$.
\end{definition}

\begin{theorem}[{\cite[Satz IV.3.4]{Ren-88}}]\label{thm:essentially-connected}
    Let $G$ be a group of type $F_k$ and let $\left[\chi\right]\in S(G)$. Then the following are equivalent:
    \begin{enumerate}
        \item $\left[\chi\right]\in {}^{\ast}\Sigma^k(G)$
        \item If $L$ is a $K(G,1)$ with finite $k$-skeleton and universal covering $D$, then for every valuation $v_{\chi}$ on $D$, the valuation subcomplex $D_v$ is essentially $(k-1)$-connected.
    \end{enumerate}
\end{theorem}

We remark that in current literature ${}^{\ast}\Sigma^k(G)$ is usually defined using the equivalent characterisation given by Theorem \ref{thm:essentially-connected}(2). For our purposes it will suffice to have a proof of Theorem \ref{thm:essentially-connected} for $n=2$. Such a proof follows from \cite[Lemma 2]{Ren-89} and the independence of essential connectivity of all choices involved in its definition \cite{Ren-88,BieStr-book}.

We finish this section by explaining the connection between BNSR-invariants and finiteness properties. For a subset $N\subseteq G$, we denote by $S(G,N)\subseteq S(G)$ the subsphere
\[
    S(G,N):= \left\{[\chi]\in S(G)\mid \chi|_N\equiv 0\right\}
\]
of characters vanishing identically on $N$. One of the main results of Renz' thesis is
\begin{theorem}[{Renz \cite[Satz C]{Ren-88}}]\label{thm:homotopical-BNSR}
    Let $G$ be a group of type $F_n$ and let $\phi: G\to Q$ be a homomorphism onto an abelian group $Q$. Then the following are equivalent:
    \begin{enumerate}
        \item $S(G,N)\subseteq {}^{\ast}\Sigma^n(G)$
        \item $N:=\ker(\phi)$ is of type $F_n$.
    \end{enumerate}
\end{theorem}

 Let us also mention that Theorem \ref{thm:homotopical-BNSR} generalises an analogous result of Bieri, Neumann, Strebel and Renz for homological BNSR invariants.
\begin{theorem}[{Bieri, Neumann, Strebel, Renz \cite{BNS-87,BieStr-book,BieRen-88}}]\label{thm:homological-BNSR}
    Let $G$ be a group of type $FP_n(\mathbb{Z})$ and let $\phi: G\to Q$ be a homomorphism onto an abelian group $Q$. Then the following are equivalent:
    \begin{enumerate}
        \item $S(G,N)\subseteq \Sigma^n(G,\Z)$
        \item $N:=\ker(\phi)$ is of type $FP_n(\Z)$.
    \end{enumerate}
\end{theorem}

Since $FP_1(\Z)$ is equivalent to $F_1$, and for $n\geq 2$ we have that being $FP_n(\Z)$ and $F_2$ implies being $F_n$, one can deduce Theorem \ref{thm:homotopical-BNSR} from Theorem \ref{thm:homological-BNSR} and the case $n=2$ of Theorem \ref{thm:homotopical-BNSR}.

\section{From the second BNSR invariant to finite presentability of kernels}\label{sec:Renz-proof}

 Our proof of Theorem \ref{thm:Main-Dehn} will be based on quantifying Renz' proof of Theorem \ref{thm:homotopical-BNSR} for $n=2$ by extracting explicit van Kampen diagrams from it and then bounding their area. In this section we therefore recall the main steps of his proof. Besides providing the background for our proof of Theorem \ref{thm:Main-Dehn}, this section may also be of independent interest, as to our knowledge it provides the first English account of some of Renz' proofs that had previously only appeared in his thesis \cite{Ren-88} (which is written in German). For this reason we present proofs of both directions of Theorem \ref{thm:homotopical-BNSR} for $n=2$, despite the fact that to prove Theorem \ref{thm:Main-Dehn} we will only require the (more difficult) direction that (1) implies (2). In writing this section we stay close to Renz' original proof in \cite{Ren-88} and will also include versions of some of the graphic illustrations he used, as they will be helpful for our Dehn function estimates. 

\subsection{Preliminary definitions and results}
\label{sec:Renz-proof-1}

We will now only be interested in the invariant ${}^{\ast}\Sigma^2(G)$ of a finitely presented group $G$. Given a presentation $\left\langle X\mid R\right\rangle$ we will denote by $C=C(X;R)$ the universal cover of its associated Cayley complex (we could complete it to a universal cover of a $K(G,1)$ by attaching higher dimensional cells, but for our purposes the 2-skeleton suffices).

We label the vertex $g\widetilde{\sigma}^0$ of $C$ by the corresponding group element $g\in G$. A base point $b$ for a van Kampen diagram $M$ is a choice of a labelling of a vertex of $M$ by an element of $b\in G$. A base point fixes a unique combinatorial map $f:M \to C$, which allows us to label the other vertices of $M$ by unique elements of $G$. For a character $\chi$ we can then define a canonical valuation $v_{\chi}$ on van Kampen diagrams by $v_{\chi}(M):=\min\left\{v_{\chi}(f(M^{0}))\right\}$, where we set $v_{\chi}(g\widetilde{\sigma}_0):=\chi(g)$. As before, we extend $v_{\chi}$ to a valuation $v_{\chi}$ on edges, 2-cells and based edge paths in $C$ by defining it to be the minimum value of $v_{\chi}$ on their set of vertices. This valuation is the discretized version of the valuation defined in the previous section (see Remark \ref{rmk:discretized-valuation}). In particular, it takes values in $\chi(G)$ and is not continuous.

Renz gives the following combinatorial characterisation of ${}^{\ast}\Sigma^1(G)$.
\begin{theorem}[{Renz \cite[Theorem 1]{Ren-89}}]\label{thm:characterisation-Sigma1}
    Let $G$ be a finitely generated group, $X$ a finite generating set and $\left[\chi\right]\in S(G)$. Then the following are equivalent:
    \begin{enumerate}
        \item $\left[\chi\right]\in {}^{\ast}\Sigma^1(G)$
        \item there is a $t\in X^{\pm 1}$ with $\chi(t)>0$ such that for every $x\in X^{\pm 1}\setminus\left\{t^{\pm 1}\right\}$ the conjugate $t^{-1}xt\in G$ can be represented by a word $w_{\chi,x,t}$ in $X^{\pm 1}$ with $v_{\chi}(t^{-1}xt)<v_{\chi}(w_{\chi,x,t})$.
    \end{enumerate}
\end{theorem}

If $r=x_1x_2\cdots x_n\in R$ is a relation with $x_i\in X^{\pm 1}$, then we write $\widehat{r}_{\chi,t}$ for the word $\widehat{r}_{\chi,t}:= w_{\chi,x_1,t}w_{\chi,x_2,t}\cdots w_{\chi,x_n,t}$. It represents a relation in $G$ and there is thus a van Kampen diagram $M_{\widehat{r}_{\chi,t}}$ with boundary $\widehat{r}_{\chi,t}$. We fix a base point $b_0$ on the van Kampen diagram with one 2-cell defined by the relation $r_{\chi,t}$ and we fix the point $b_1:=b_0\cdot t$ as base point on $\widehat{r}_{\chi,t}$ for the diagram $M_{\widehat{r}_{\chi,t}}$. In particular, our choice of labelled vertex in $M_{\widehat{r}_{\chi,t}}$ naturally corresponds to a $t$-translate of our choice of labelled vertex in the van Kampen diagram for $r$. Then $v(b_1)=v(b_0)+\chi(t)$. \vspace{.2cm}

\noindent {\bf Convention.} When it is clear from context which $\chi$ and $t$ we chose, we will subsequently write $\widehat{r}$ (resp. $M_{\widehat{r}}$) for $\widehat{r}_{\chi,t}$ (resp. $M_{\widehat{r}_{\chi,t}}$) to simplify notation.

\begin{remark}\label{rmk:suitable-presentation}
Condition (2) in Theorem \ref{thm:characterisation-Sigma1} is open in $\left[\chi\right]$. If $A\subseteq {}^{\ast}\Sigma^1(G)$ is a closed (and thus compact) subset, then we can fix a finite set $\mathfrak{S}$ of words of the form $w_{\chi,x,t}$ with the following property:

\emph{For every $\left[\chi'\right]\in A$ there is a $t\in X^{\pm 1}$ with $\chi'(t)>0$ and a $\left[\chi\right]\in A$ such that for every $x\in X^{\pm 1}\setminus \left\{t^{\pm 1}\right\}$ we can choose $w_{\chi',x,t}=w_{\chi,x,t}\in \mathfrak{S}$ in Theorem \ref{thm:characterisation-Sigma1}.}

If, in addition, $G$ is finitely presented with presentation $\left\langle X\mid R\right\rangle$, we may thus always assume that $R$ contains all relations of the form $t^{-1}xt=w_{\chi,x,t}$ with $w_{\chi,x,t}\in \mathfrak{S}$.
\end{remark}

We can then give the following characterisation of the BNSR-invariant ${}^{\ast}\Sigma^2(G)$:
\begin{theorem}\label{thm:characterisation-Sigma2}
    Let $G$ be a finitely presented group and assume that we have chosen a presentation $\left\langle X\mid R\right\rangle$ as in Remark \ref{rmk:suitable-presentation} for a closed subset $A\subseteq {}^{\ast}\Sigma^1(G)$. 
    Then the following are equivalent for a character $\left[\chi\right]\in A$:
    \begin{enumerate}
        \item $\left[\chi\right]\in{}^{\ast}\Sigma^2(G)$.
        \item For every closed edge path $p$ in $C_v$ there is a van Kampen diagram $M$ such that $\partial M = p$ and $v_{\chi}(M)\geq 0$.
        \item There is a generator $t\in X^{\pm 1}$ with $\chi(t)>0$ such that for every relation $r\in R^{\pm 1}$ there is a van Kampen diagram $M_{\widehat{r}}$ for $\widehat{r}$ with $\partial M_{\widehat{r}} = \widehat{r}$ and $v_{\chi}(M_{\widehat{r}})> v_{\chi}(r)$.
    \end{enumerate}
\end{theorem}
\begin{proof}
    The equivalence of (1) and (2) is \cite[Theorem 2]{Ren-89}, while the equivalence of (1) and (3) is \cite[Theorem 3]{Ren-89} and \cite[Lemma 3]{Ren-89}.
\end{proof}

\subsection{{Proof of Theorem \ref{thm:homotopical-BNSR} for \texorpdfstring{$n=2$}{n=2}}}
\label{sec:Renz-proof-2}

Now assume that $\pi:G\to Q$ is a homomorphism onto an abelian group and denote by $N:=\ker(\phi)$ its kernel. We may assume that $Q= \mathbb{Z}^n$ is free abelian by passing to finite index subgroups of $G$ and $N$ if necessary. This will not change the finiteness properties of $N$.  

There is an isomorphism $\R^n\cong {\rm Hom}(Q,\mathbb{R})$ which maps $x\in \mathbb{R}^n$ to the homomorphism $q\to \langle x,q\rangle$ for $\langle , \rangle$ the Euclidean inner product on $\R^n$. We observe that the map $\pi$ induces a map $\pi^{\ast}:S(Q)\hookrightarrow S(G)$ with image $S(G,N)$. Moreover, the restriction of the Euclidean norm on $\mathbb{R}^n$, equips $Q=\mathbb{Z}^n$ with a Euclidean norm. This induces a \emph{norm on $G$} by setting $||g||:=||\pi(g)||$ with the following properties:
\begin{itemize}
    \item $||gh||\leq ||g|| + ||h||$ for all $g,h\in G$
    \item $||ng|| = ||g||$ for all $g\in G$, $n\in N$.
\end{itemize}

As above, given a based van Kampen diagram (resp. a path starting in a point of $G$), we can define an induced norm by maximising the norm $||\cdot ||$ over all labels of the diagram. \emph{We emphasize that here we take the maximum of all norms of labels rather than the minimum which we took for the valuation.}

The easier direction of the proof of Theorem \ref{thm:homotopical-BNSR} will be that we can deduce from $N$ being of type $F_2$ that $S(G,N)\subseteq {}^{\ast}\Sigma^2(G)$. To prove this we will use the equivalence of (1) and (3) in Theorem \ref{thm:characterisation-Sigma2}. The converse direction will require more work. We can use the finite presentability of $G$ to construct a van Kampen diagram for a given edge loop. We then push this filling down to a filling in a bounded neighbourhood of $N$ with respect to the norm $||\cdot ||$ on $G$. Theorem \ref{thm:characterisation-Sigma2}(3) will ensure that we can do this. We will deduce from this, that a closed bounded neighbourhood of $N$ in $C$ with respect to our norm is connected and simply connected. Since $N$ acts freely and cocompactly on this neighbourhood, this will imply that $N$ is finitely presented.

\begin{notation*}
For the remainder of this note, for $r\geq 0$, we will write $C_r$ for the full subcomplex of $C$ induced by the $0$-cells $g\widetilde{\sigma}^0$ with $||g\widetilde{\sigma}^0||\leq r$. This differs from the notation $C_{v,r}$ used in Sections \ref{sec:BNSR-invariants} and \ref{sec:Renz-proof-1}. However, no confusion should arise, as the only thing from those sections that we will use subsequently is the equivalence of (1) and (3) in Theorem \ref{thm:characterisation-Sigma2}.
\end{notation*}

\begin{lemma}\label{lem:cocompact-action}
    For every $r\geq 0$ the CW-complex $C_r$ is a free $N$-complex and the $N$-action on $C_r$ is cocompact.
\end{lemma}
\begin{proof}
    It is clear that $N$ acts freely on $C_r$. Moreover, the ball of radius $r$ with respect to the norm $||\cdot||$ on $Q$ has finitely many elements. Its preimage in $G$ thus consists of finitely many $N$-orbits. Since $C$ is locally finite, this implies that $N\backslash C_r$ is compact.
\end{proof}
A direct consequence of Lemma \ref{lem:cocompact-action} is:
\begin{corollary}\label{cor:finpres}
    If there is a $q\geq 0$ such that the subcomplex $C_q$ is connected and simply connected, then $N$ is finitely presented.
\end{corollary}

We now proceed to the proof of Theorem \ref{thm:homotopical-BNSR} and we start with the proof that (2) implies (1). Assume that $N$ is finitely presentable and let $N\cong \left\langle x_1,\cdots,x_m\mid r_1,\cdots, r_k\right\rangle$ be a finite presentation. Fix a finite presentation $\left\langle y_1,\cdots, y_n\mid s_1,\cdots , s_l\right\rangle$ for $Q$, where the $y_i$ correspond to a standard generating set of $\Z^n$ and the $s_j$ are the standard commutator relations between the $y_i$.

Since $G$ is an extension of $N$ by $Q$ we can choose a presentation for $G$ with generating set $\left\{x_1,\cdots, x_m, y_1, \cdots, y_n\right\}$ and relations $R\cup S_1 \cup S_2$, where
\begin{align*}
    R &= \left\{r_i(X)\mid 1\leq i \leq k\right\},\\
    S_1 &= \left\{ s_j(Y)=w_j(X) \mid 1\leq j \leq l \right\},\\
    S_2 &= \left\{y_j^{-1}x_i y_j= u_{ij}(X),~ y_jx_iy_j^{-1}=v_{ij}(X)\mid 1\leq i \leq m,~ 1\leq j \leq n\right\}.\\
\end{align*}

Let $\left[\chi\right]\in S(G,N)$. Since $N$ is finitely generated, we have $\left[\chi\right]\in {}^{\ast}\Sigma^1(G)$, by Theorem \ref{thm:homotopical-BNSR} for $n=1$. We can also prove this directly using Theorem \ref{thm:characterisation-Sigma1}. Indeed, for a generator $t$ with $\chi(t)>0$ we have $t\in \left\{y_1,\dots,y_n\right\}^{\pm 1}$, say $t=y_j$, and thus
\begin{align*}
    &t^{-1}x_it=u_{ij}(X) \mbox { and } t^{-1}x_i^{-1}t=v_{ij}(X)^{-1} & \mbox{ for } 1\leq i \leq m, \mbox{ and }\\
    &t^{-1}y_it=y_iw_{k_{ij}}(X) \mbox{ and } t^{-1} y_i^{-1} t = w_{k_{ij}}(X)^{-1} y_i^{-1} &\mbox{ for } 1\leq j\leq n,~i\neq j,~1\leq k_{ij}\leq n,\\
\end{align*}
showing that Theorem \ref{thm:characterisation-Sigma1} holds for any choice of $t$ with $\chi(t)>0$. Note that this also shows that our chosen presentation has the form of Remark \ref{rmk:suitable-presentation} for $A=S(G,N)$. 

To prove that $\left[\chi\right]\in {}^{\ast}\Sigma^2(G)$, we fix an element $t\in \left\{y_1^{\pm 1},\cdots, y_m^{\pm 1}\right\}$ with $\chi(t)>0$, say $t=y_1$, and verify that Condition (3) in Theorem \ref{thm:characterisation-Sigma2} holds for each of the relations of type $R$, $S_1$ and $S_2$.

If $r\in R$, then $t^{-1}x_it=u_{i1}(X)$ for $1\leq i \leq m$. In particular, $\widehat{r}$ is a word in the generators $X$ of $N$, implying that there is a van Kampen diagram $M_{\widehat{r}}$ with relations from $R$. Choosing a base point for $M_{\widehat{r}}$ which is a right-$t$-translate of a base point for $r$, we deduce that $v_{\chi}(M_{\widehat{r}})>v_{\chi}(r)$. Thus, Condition (3) holds for relations from $R$. 

One readily verifies that one can also construct explicit van Kampen diagrams $M_{\widehat{s}_i}$ for elements $s_i\in S_i$ with $v_{\chi}(M_{\widehat{s}_i})>v_{\chi}(s_i)$ to show that Condition (3) also holds for relations from $S_1$ and $S_2$ (see Figures V.1 - V.3 in \cite{Ren-88}). By Theorem \ref{thm:characterisation-Sigma2} this implies that $\left[\chi\right]\in {}^{\ast}\Sigma^2(G)$.

Conversely, assume that $S(G,N)\subseteq {}^{\ast}\Sigma^2(G)$. Since ${}^{\ast}\Sigma^2(G)\subseteq {}^{\ast}\Sigma^1(G)$, this implies that $N$ is finitely generated. We can choose a presentation $\left\langle X \mid R\right\rangle$ for $G$ such that the set $C_q$ is connected for all $q\geq 0$ by choosing a generating set $X$ which contains a finite generating set for $N$. We may further assume, by adding relations if necessary, that this presentation satisfies the properties of Remark \ref{rmk:suitable-presentation} with $A=S(G,N)$.

By Theorem \ref{thm:characterisation-Sigma2}(3), for every $\left[\chi\right]\in S(G,N)$ there is a generator $t\in X^{\pm 1}$ with $\chi(t)>0$ and a finite set $\mathfrak{M}_{\chi,t}$ of van Kampen diagrams such that for every $r\in R$ there is precisely one $M_{\widehat{r}}\in \mathfrak{M}_{\chi,t}$ with $\partial M_{\widehat{r}} = \widehat{r}$, and, moreover, this $M_{\widehat{r}}$ satisfies $v_{\chi}(M_{\widehat{r}})>v_{\chi}(r)$. Since this condition is open in $S(G,N)$, there is a neighbourhood $U_{\chi,t}$ of $\left[\chi\right]$ in $S(G,N)$ and a constant $c_{\chi,t}>0$ such that for every $\left[\chi'\right]\in U_{\chi,t}$, for $t'=t$ and for $\mathfrak{M}_{\chi',t'}:=\mathfrak{M}_{\chi,t}$ we obtain
\[
    \min \left\{ v_{\chi'}(M_{\widehat{r}})-v_{\chi'}(r)\mid M_{\widehat{r}}\in \mathfrak{M_{\chi,t}}, r\in R\right\}\geq c_{\chi,t}.
\]

By compactness of $S(G,N)$ we can cover it by finitely many of the $U_{\chi,t}$. Thus, there is a finite set $\mathfrak{M}$ and a constant $a>0$ such that for every $\left[\chi\right]\in S(G,N)$ there is a generator $t\in X^{\pm 1}$ and a subset $\mathfrak{M}_{\chi,t}\subset \mathfrak{M}$ with the following properties:
\begin{enumerate}
    \item for every $r\in R$ there is precisely one $M_{\widehat{r}}\in \mathfrak{M}_{\chi,t}$ with $\partial M_{\widehat{r}}=\widehat{r}$;
    \item $\min \left\{v_{\chi}(M_{\widehat{r}})-v_{\chi}(r)\mid r\in R, M_{\widehat{r}}\in \mathfrak{M}_{\chi,t}\right\}\geq a$.
\end{enumerate}

The homomorphism $\pi: G\to Q$ identifies $S(G,N)$ with the unit sphere $S^{n-1}$ in $\R^{n}$. For $u\in S^{n-1}$ we denote by $\chi_u:G\to \R$ the corresponding character and by $v_u$ the corresponding valuation. 

Aside from $a$ we will require a second real parameter that we will now introduce. The valuation $v_u(M_{\widehat{r}})$ and the norm $||M_{\widehat{r}}||$ depend on our choice of base point for $r$. We denote by $\mathfrak{M}'$ the set of based van Kampen diagrams obtained from the van Kampen diagrams $M_{\widehat{r}}\in \mathfrak{M}$ by choosing as base point a right $t$-translate of every possible base point of $r$ (where we label the chosen base point of $r$ with a group element of valuation $0$). We define
\[
    b:=\max\left\{||M_{\widehat{s}}||\mid M_{\widehat{s}}\in \mathfrak{M}'\right\}.
\]

The proof that (1) implies (2) in Theorem \ref{thm:homotopical-BNSR} then follows from the following result.

\begin{lemma}\label{lem:van-Kampen-diagrams}
    Let $q> \max\left\{\frac{b^2}{a},a\right\}$. Then for every closed edge path $p$ with base point $1$ in $C_q$ there is a van Kampen diagram $M_p$ for $p$ with $||M_p||\leq \max\left\{||p||, \frac{b^2}{a}, a\right\}$.
\end{lemma}

Before proving Lemma \ref{lem:van-Kampen-diagrams}, we explain how it implies Theorem \ref{thm:homotopical-BNSR}(2). So assume that Lemma \ref{lem:van-Kampen-diagrams} holds. Then for every $q>\max\left\{\frac{b^2}{a},a\right\}$ and every closed edge path $p$ in $C_q$ based at $1$ there is a van Kampen diagram $M_p$ with $||M_p||\leq \max\left\{||p||,\frac{b^2}{a},a\right\}\leq q$. Thus, the connected component of $C_q$ containing $1$ is simply connected. Our choice of presentation for $G$ guarantees that $C_q$ is connected. Thus, we deduce from Lemma \ref{cor:finpres} that $N$ is finitely presented, implying that Theorem \ref{thm:homotopical-BNSR}(2) holds.

\begin{proof}[{Proof of Lemma \ref{lem:van-Kampen-diagrams}}]
    Let $q>\max\left\{\frac{b^2}{a},a\right\}$ and let $p$ be a closed edge path in $C_q$. There is a van Kampen diagram $M_p$ for $p$ in $C$. Assume that $||M_p||>\max\left\{||p||,\frac{b^2}{a},a\right\}$,  else we are done. Then there is at least one interior point $g$ of $M_p$ with $c:= ||g||=||M_p||\in \R_{>0}$. Let $u:= \frac{\pi(g)}{||g||}$, and let $\chi_u:G\to \R$ be the corresponding homomorphism. We assume that $\chi_u(g)=-c$ attains its minimum in $g$ (if $\chi_u(g)=+c$ we replace $\chi_u$ by $-\chi_u$). By our assumptions on $\mathfrak{M}$, we can choose a $t\in X^{\pm 1}$ with $\chi_u(t)>0$, a valuation $v_u$, and a subset $\mathfrak{M}_{\chi_u,t}\subset \mathfrak{M}$ such that the following hold:
    \begin{itemize}
        \item[(i)] for every relation $r\in R$ there is a unique van Kampen diagram $M_{\widehat{r}}\in \mathfrak{M}_{\chi_u,t}$ with $\partial M_{\widehat{r}}=\widehat{r}$;
        \item[(ii)] $ \min\left\{v_u(M_{\widehat{r}})-v_u(r)\mid M_{\widehat{r}}\in \mathfrak{M}_{\chi_u,t}\right\}\geq a$
    \end{itemize}
    We replace the star of the vertex $g$ in the diagram $M_p$ by a new filling which reduces the norm of $g$ by at least $a$ and all new vertices created in the process have norm at most $c-\frac{a}{2}$. We do this by performing the modifications pictured in Figures \ref{fig:pushing-1}, \ref{fig:pushing-2} and \ref{fig:pushing-3} (based on Figures IV.2, IV.3 and IV.4 in \cite{Ren-88}). To obtain the diagram in Figure \ref{fig:pushing-2} from the diagram in Figure \ref{fig:pushing-1} we replace 2-cells labelled by relations $r_i$ adjacent to $g$ by a $t$-ring with a van Kampen diagram of type $M_{\widehat{r}_i}\in \mathfrak{M}_{\chi_u,t}$ glued along its interior. Our choice of presentation guarantees that the words of the form $t^{-1}xt=w_{\chi,x,t}$ making up the $t$-rings are relations from $R$. This allows us to obtain the diagram in Figure \ref{fig:pushing-3} from the diagram in Figure \ref{fig:pushing-2} by cancelling adjacent pairs of relations labelled $s_i$ and $s_i^{-1}$. As a result we obtain a new van Kampen diagram $M'_p$ for $p$ in which by (ii) all new vertices $h$ satisfy $v_u(h)>-c+a$. Moreover, by definition of $b$ the images of the new vertices in $Q$ have distance at most $b$ from $\pi(g)$. 
    
    Equivalently, using the definition of $\chi_u$, all images of the new vertices under $\pi$ lie in the intersection 
    \[
        \left\{y\in \Z^n \mid \left\langle \frac{\pi(g)}{||g||}, y\right\rangle \geq -c + a\right\}\cap\left\{y\in \Z^n\mid ||y-\pi(g)||\leq b  \right\}.
    \]
    
    An easy calculation using Pythagoras' Theorem (compare \cite[Figure V.4]{Ren-88}) shows that the norm of every new vertex $y$ satisfies\footnote{Renz' proof differs slightly in that he only assumes $c>\frac{b^2}{2a}$, which implies $||y||<c$. Since the $||y||$ are square roots of integers, this process will still terminate. However, an additional argument is required to see that the norm reduction is by a uniform constant that does not depend on $c$, a property we will need in Section \ref{sec:Dehnfunctions}.}
    \begin{align*}
        ||y|| &\leq \sqrt{(c-a)^2+(b^2-a^2)} =\sqrt{c^2-2ac+b^2}\\
        &<\sqrt{c^2-2ac+ac}<\sqrt{c^2-ac +\frac{a^2}{4}}=c-\frac{a}{2},
    \end{align*}
    where for the second inequality we use that $c>\frac{b^2}{a}$ and for the last equality we use that $c>a$.
    
    We deduce that $||M_p'||\leq ||M_p||$ and the number of vertices of $M_p'$ with norm $\geq c-\frac{a}{2}$ has been reduced by $1$. Since $q>\max\left\{\frac{b^2}{a},a\right\}$, we can inductively reduce every van Kampen diagram in $C$ for a loop $p$ in $C_q$ to a van Kampen diagram for $p$ in $C_q$ (in finitely many steps). We deduce that $C_q$ is simply connected. Thus, Lemma \ref{lem:cocompact-action} implies that $N$ is finitely presented.
\end{proof}

\begin{figure}
    \begin{center}
    \def\svgwidth{9.5cm}
    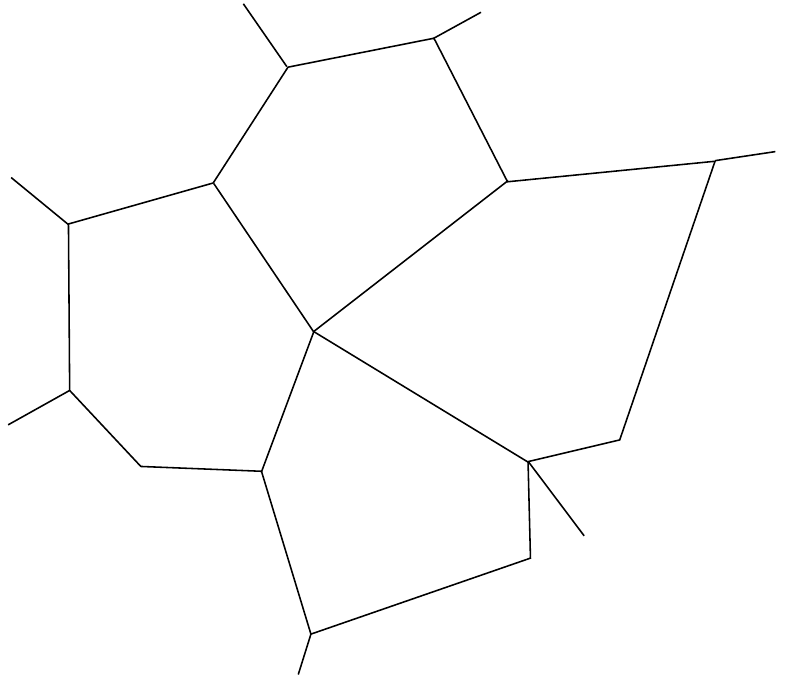
    \end{center}
    \caption{~}
    \label{fig:pushing-1}
\end{figure}

\begin{figure}
    \begin{center}
    \def\svgwidth{9.5cm}
    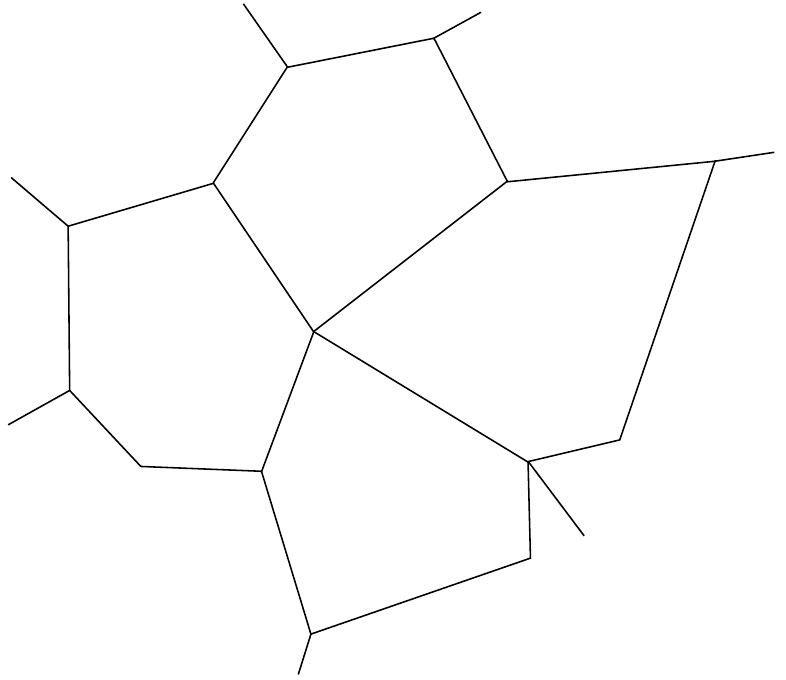
    \end{center}
    \caption{~}
    \label{fig:pushing-2}
\end{figure}

\begin{figure}

    \begin{center}
    \def\svgwidth{9.5cm}
    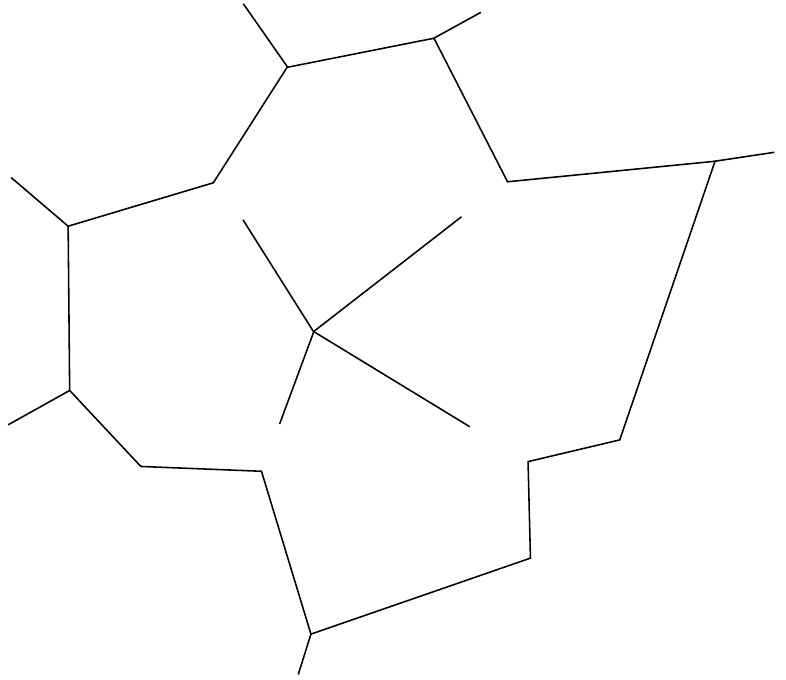
    \end{center}
    \caption{~}
    \label{fig:pushing-3}
\end{figure}

\section{Dehn functions of coabelian subgroups}\label{sec:Dehnfunctions}

The goal of this section is to prove Theorem \ref{thm:Main-Dehn}, generalising \cite[Theorem B]{GerSho-02} by Gersten and Short, and then explain how it implies Theorem \ref{thm:Dehn-hyp} and Corollary \ref{cor:Dehn-automatic}.

Assume that $G$ is a finitely presented group with finite presentation $\mathcal{P}=\left\langle X\mid R\right\rangle$. For words $w(X)$ and $v(X)$ in $X^{\pm 1}$ we write $w(X)=_{F_{X}}v(X)$ if they represent the same elements of the free group in $X$. We call a word $w(X)$ \emph{null-homotopic} if it represents the trivial element of $G$ and define its \emph{area}
\[
    {\rm Area}_{\mathcal{P}}(w):= \min\left\{k\mid w(X)=_{F_X}\prod_{i=1}^k u_i(X)\cdot r_i\cdot u_i(X)^{-1},~ r_i\in R^{\pm 1},~ u_i(X)\in F_X\right\}.
\]

The \emph{Dehn function} of $G$ with respect to $\mathcal{P}$ is defined as
\[
    \delta_{G,\mathcal{P}}(n):=\max\left\{{\rm Area}_{\mathcal{P}}(w)\mid w \mbox{ null-homotopic, } \ell(w)\leq n\right\}.
\]
It measures the maximal number of conjugates of relations needed to detect if a given word in $X^{\pm 1}$ represents the trivial element of $G$. Up to asymptotic equivalence\footnote{Recall that for two functions $f,g:\mathbb{R}_{>0}\to \mathbb{R}_{>0}$ we say that $g$ \emph{asymptotically bounds} $f$ and write $f\preccurlyeq g$ if there is a constant $C\geq 1$ such that $f(n)\leq Cg(Cn+C)+Cn+C$. We say that $f$ and $g$ are \emph{asymptotically equivalent} and write $f\asymp g$ if $f\preccurlyeq g \preccurlyeq f$. } $\delta_{G,\mathcal{P}}$ is independent of the choice of finite presentation and we will thus often write $\delta_G$ instead of $\delta_{G,\mathcal{P}}$ when we are only interested in the asymptotic equivalence class of the Dehn function. 

We can describe any filling of a null-homotopic word $w(X)$ by a van Kampen diagram with boundary word $w$. The area of a van Kampen diagram is its number of 2-cells and ${\rm Area}_{\mathcal{P}}(w)$ is then equal to the number of 2-cells in a van Kampen diagram of minimal area for $w$; see \cite{Bri-02} for the definition of a van Kampen diagram. 

The \emph{radius} ${\rm Radius}_{\mathcal{P}}(M)$ of a van Kampen diagram $M$ is the maximal distance of one of its vertices from its boundary. We call a pair of functions $(f,g):\mathbb{N}\to \mathbb{R}^2$ an \emph{area-radius pair} (or \emph{AR pair}) if for every word $w$ of length $\ell(w)\leq n$ there is a van Kampen diagram of area at most $f(n)$ and of radius at most $g(n)$. By \cite[Proposition 2.1]{GerSho-02} area-radius pairs are independent of the choice of finite presentation of a group up to a suitable asymptotic equivalence of functions; here the equivalence relation on the area functions is given by $\asymp$, while the equivalence relation $\simeq$ on the radius functions is induced by $g_1\lesssim g_2$ if there is a $C>0$ with $g_1(n)\lesssim C g_2(Cn) + C$.

The main idea of the proof of Theorem \ref{thm:Main-Dehn} is that it is possible to control the area of the van Kampen diagrams constructed in Section \ref{sec:Renz-proof} for a null-homotopic word in $N$ in terms of the area and radius of the initial van Kampen diagram for this word in $G$. The main step for this is the following result:
\begin{lemma}\label{lem:van-Kampen-diagrams-quantitative}
    Let $G$ and $N$ be as in Lemma \ref{lem:van-Kampen-diagrams} and let $q>\max\left\{\frac{b^2}{a},a\right\}$. Then there is a constant $A>1$ such that for every closed edge path $p$ in $C_q$ with base point $1$ and every van Kampen diagram $M_p$ for $p$ with $||M_p||=:c>q$ there is a van Kampen diagram $M_p'$ for $p$ with $||M_p'||\leq \max\left\{ c-\frac{a}{2}, q\right\}$ and
    \[
        {\rm Area}(M_p')\leq A\cdot {\rm Area}(M_p).
    \]
\end{lemma}

Before proving Lemma \ref{lem:van-Kampen-diagrams-quantitative} we explain how Theorem \ref{thm:Main-Dehn} follows from it.

\begin{proof}[{Proof of Theorem \ref{thm:Main-Dehn}}]
    Since the Euclidean norm on $Q$ is bounded above by a constant times the word metric for our chosen generating set, the map $\pi: G\to Q$ is Lipschitz with respect to the word metric $d_G$ on $G$ and the (Euclidean) norm $||\cdot||$ on $Q$. Thus, there is a constant $B>0$ such that for all elements $g,h\in G$ we have 
    \begin{equation}\label{eqn:Lipschitz-map}
    ||\pi(g)-\pi(h)||\leq B \cdot d_G(g,h).
    \end{equation}

    Let $q>\max\left\{\frac{b^2}{a},a\right\}$ and let $p$ be an arbitrary edge loop of length $n$ in $C_q$ based at $1\in G$. By our assumptions we can choose a van Kampen diagram $M_p$ for $p$ with ${\rm Area}(M_p)\leq f(n)$ and ${\rm Radius}(M_p)\leq g(n)$. In particular, it follows from \eqref{eqn:Lipschitz-map} and our assumption on $p$ that 
    \[
        ||M_p||\leq q + B\cdot {\rm Radius}(M_p)\leq q + B\cdot g(n).
    \]
    
    Inductively applying Lemma \ref{lem:van-Kampen-diagrams-quantitative} at most $\frac{2\cdot B \cdot g(n)}{a}$ times, we deduce that there is a van Kampen diagram $\overline{M}_p$ for $p$ in $C_q$ with
    \[ 
        {\rm Area}(\overline{M}_p)\leq A^{\frac{2\cdot B \cdot g(n)}{a}}\cdot {\rm Area}(M_p)\leq \left(A^{\frac{2B}{a}}\right)^{g(n)} \cdot f(n).
    \]
    
    Since $N$ acts properly, cocompactly and cellularly on $C_q$, we deduce that for every finite presentation of $N$ there is a constant $A'>1$ such that
    \[
        \delta_N(n)\leq A'^{g(n)}f(n).
    \]  
    This completes the proof.\footnote{Note that $\overline{M}_p$ is a van Kampen diagram for $p$ in our chosen presentation for $G$. To conclude we use the well-known and not hard to see fact that the Dehn function of a finitely presented group is asymptotically equivalent to the combinatorial filling function of any simply connected 2-complex on which this group acts properly cocompactly and cellularly. Here the combinatorial filling function of a 2-complex is the maximal combinatorial filling area for an edge loop of length at most $n$, and the combinatorial filling area of an edge loop in this 2-complex is the minimal number of 2-cells that a null-homotopy has to traverse.}
    
\end{proof}

We will now prove Lemma \ref{lem:van-Kampen-diagrams-quantitative}.

\begin{proof}[{Proof of Lemma \ref{lem:van-Kampen-diagrams-quantitative}}]
    The proof of Lemma \ref{lem:van-Kampen-diagrams} provides us with a procedure for replacing a van Kampen diagram $M$ with norm $c>\max\left\{\frac{b^2}{a},a\right\}$ by a new van Kampen diagram $M'$ with one vertex less of norm $> c-\frac{a}{2}$. More precisely, we choose a vertex $v\in M^0$ of maximal norm and replace the star of $v$ in $M$ by a new filling with the same boundary such that all interior vertices of this new filling have norm $\leq c-\frac{a}{2}$. This adds new vertices, but does not alter any of the other vertices of the initial van Kampen diagram $M$. It is clear from the proof of Lemma \ref{lem:van-Kampen-diagrams} that the area of $M'$ only depends on the degree ${\rm deg}_{M}(v)$ of the vertex $v$ in $M$, the chosen presentation for $G$ and the areas of the van Kampen diagrams in the finite set $\mathfrak{M}$ in the proof of Lemma \ref{lem:van-Kampen-diagrams} (which only depend on $G$ and the homomorphism $\pi:G\to Q$). In particular, there is a constant $A>0$ such that
    \[
       {\rm Area}(M')\leq {\rm Area}(M) + A \cdot {\rm deg}_M(v). 
    \]
    
    Let now $p$ be a closed edge loop in $C_q$ and let $M_p$ be a van Kampen diagram for $p$ in $C$ with norm $c:=||M_p||$. We apply the procedure from Lemma \ref{lem:van-Kampen-diagrams} iteratively to construct a sequence $M_{p,i}$, $1\leq i \leq k < |M_p^0|$, of van Kampen diagrams for $p$ such that $M_{p,i}$ is obtained from $M_{p,i-1}$ by replacing the star of a vertex $v_i$ of maximal norm $||v_i||>\max\left\{c-\frac{a}{2},q\right\}$ from $M_{p,i-1}$ by a new filling. Since $p$ is a loop in $C_q$ this vertex is necessarily an interior vertex and, moreover, it can naturally be identified with a vertex $v_i$ in each of the diagrams $M_{p,j}$ with $j\leq i-1$. Observe that for every vertex $v\in M_p^0$ that survives in $M_{p,i-1}$ we have ${\rm deg}_{M_{p,i-1}}(v)\leq 2{\rm deg}_{M_p}(v)$. Indeed, the degree of $v$ only changes when passing from $M_{p,j-1}$ to $M_{p,j}$ if some of its adjacent 2-cells contain the vertex $v_j$. There are two possible cases for such adjacent 2-cells:
    \begin{enumerate}
        \item The 2-cell is a cell that survived from the original diagram $M_p$, in which case it gets replaced by at most two new 2-cells containing $v$ (see Figure \ref{fig:replacing-original-cell}), increasing the degree of $v$ by at most one.
        \item The 2-cell is a cell that was produced when passing from $M_{p,\ell-1}$ to $M_{p,\ell}$ for some $\ell<j$. In this case the vertex $v_j$ must be connected to the vertex $v$ by an edge and therefore there is a second 2-cell adjacent to $v$ which also contains this edge. One observes that the degree of $v$ does not increase when we replace these two 2-cells when passing from $M_{p,j-1}$ to $M_{p,j}$ (see Figure \ref{fig:replacing-new-cell}).
    \end{enumerate}
    Since Case (1) can only occur at most ${\rm deg}_{M_p}(v)$ many times, the assertion that ${\rm deg}_{M_{p,i-1}}(v)\leq 2{\rm deg}_{M_p}(v)$ follows.
    
    \begin{figure}
        \begin{center}
            \def\svgwidth{14cm}
            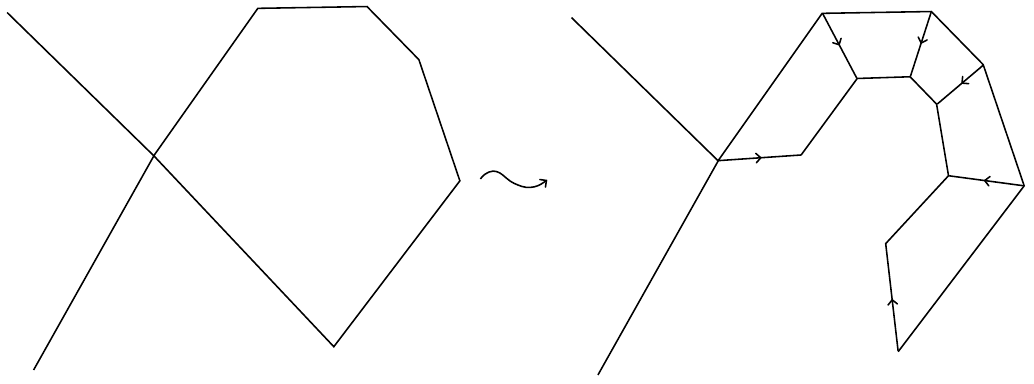
        \end{center}
            \caption{~}
        \label{fig:replacing-original-cell}
    \end{figure}
    
    \begin{figure}
        \begin{center}
            \def\svgwidth{14.5cm}
            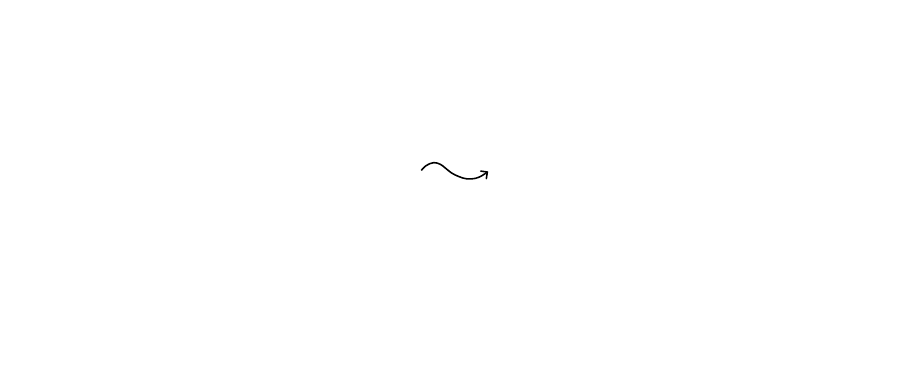
        \end{center}
            \caption{~}
        \label{fig:replacing-new-cell}
    \end{figure}
    
    By construction, the diagram $M_p':=M_{p,k}$ satisfies $||M_p'||\leq \max\left\{ c- \frac{a}{2},q\right\}$ and 
    \begin{equation}\label{eqn:area-degree-estimate-1}
        {\rm Area}(M_p')\leq {\rm Area}(M_p) + 2\cdot A\cdot  \sum_{v {\rm ~interior~ vertex~of~} M_p} {\rm deg}_{M_p}(v).
    \end{equation}
    
    Since $G$ is finitely presented there is a constant $B>0$ such that every 2-cell of a van Kampen diagram in $C$ has at most $B$ edges on its boundary. Moreover, the sum of the degrees of the interior vertices of $M_p$ is bounded above by two times the number of interior edges of $M_p$. We deduce that 
    \begin{equation}\label{eqn:area-degree-estimate-2}
        \sum_{v {\rm ~interior~ vertex~of~} M_p} {\rm deg}_{M_p}(v) \leq 2 \cdot B \cdot {\rm Area}(M_p).
    \end{equation}
    
    Combining the estimates \eqref{eqn:area-degree-estimate-1} and \eqref{eqn:area-degree-estimate-2} we deduce that 
    \[
        {\rm Area}(M_p')\leq \left(1 + 4\cdot A\cdot B\right) {\rm Area}(M_p).
    \]
    
    Since $A$ and $B$ only depend on the group $G$ and our choice of the finite set $\mathfrak{M}$ for a suitable finite presentation of $G$, this completes the proof.
\end{proof}

The other results in the introduction are straight-forward consequences of Theorem \ref{thm:Main-Dehn} using the same arguments as in \cite{GerSho-02}.
\begin{proof}[{Proof of Theorem \ref{thm:Dehn-hyp}}]
    By \cite[Lemma 2.2]{GerSho-02} hyperbolic groups have area-radius pairs of the form $n\mapsto (A n\log(n),B \log(n))$ for suitable constants $A,~B>0$. Theorem \ref{thm:Dehn-hyp} is thus an immediate consequence of Theorem \ref{thm:Main-Dehn}.
\end{proof}

\begin{proof}[{Proof of Corollary \ref{cor:Dehn-automatic}}]
    As explained on page 69 of \cite{GerSho-02}, if $G$ is a synchronously (resp. asynchronously) automatic group, then it has an area-radius pair of the form $(A n^2, B n)$ for constants $A,~B>0$ (resp. $(C^n,D n)$ for constants $C>1$ and $D>0$). The assertion then follows from Theorem \ref{thm:Main-Dehn}.
\end{proof}

\frenchspacing
\bibliography{References}
\bibliographystyle{amsplain}

\end{document}

%% file: figure-pushing-1-inkscape.pdf_tex
%% Creator: Inkscape 1.3 (0e150ed6c4, 2023-07-21), www.inkscape.org
%% Accompanies image file 'figure-pushing-1-inkscape.pdf' (pdf, eps, ps)
\begingroup%
  \makeatletter%
  \providecommand\color[2][]{%
    \errmessage{(Inkscape) Color is used for the text in Inkscape, but the package 'color.sty' is not loaded}%
    \renewcommand\color[2][]{}%
  }%
  \providecommand\transparent[1]{%
    \errmessage{(Inkscape) Transparency is used (non-zero) for the text in Inkscape, but the package 'transparent.sty' is not loaded}%
    \renewcommand\transparent[1]{}%
  }%
  \providecommand\rotatebox[2]{#2}%
  \newcommand*\fsize{\dimexpr\f@size pt\relax}%
  \newcommand*\lineheight[1]{\fontsize{\fsize}{#1\fsize}\selectfont}%
  \ifx\svgwidth\undefined%
    \setlength{\unitlength}{377.00787402bp}%
    \ifx\svgscale\undefined%
      \relax%
    \else%
      \setlength{\unitlength}{\unitlength * \real{\svgscale}}%
    \fi%
  \else%
    \setlength{\unitlength}{\svgwidth}%
  \fi%
  \global\let\svgwidth\undefined%
  \global\let\svgscale\undefined%
  \makeatother%
  \begin{picture}(1,0.86466165)%
    \lineheight{1}%
    \setlength\tabcolsep{0pt}%
    \put(0,0){\includegraphics[width=\unitlength,page=1]{figure-pushing-1-inkscape.pdf}}%
    \put(0.3960717,0.47169784){\makebox(0,0)[lt]{\lineheight{1.25}\smash{\begin{tabular}[t]{l}$g$\end{tabular}}}}%
    \put(0.20311496,0.42917931){\makebox(0,0)[lt]{\lineheight{1.25}\smash{\begin{tabular}[t]{l}$r_2$\end{tabular}}}}%
    \put(0.42847701,0.65086697){\makebox(0,0)[lt]{\lineheight{1.25}\smash{\begin{tabular}[t]{l}$r_1$\end{tabular}}}}%
    \put(0.65155219,0.46533708){\makebox(0,0)[lt]{\lineheight{1.25}\smash{\begin{tabular}[t]{l}$r_4$\end{tabular}}}}%
    \put(0.46302801,0.22406413){\makebox(0,0)[lt]{\lineheight{1.25}\smash{\begin{tabular}[t]{l}$r_3$\end{tabular}}}}%
  \end{picture}%
\endgroup%

%% file: figure-pushing-2-inkscape.pdf_tex
%% Creator: Inkscape 1.3.2 (091e20e, 2023-11-25, custom), www.inkscape.org
%% Accompanies image file 'figure-pushing-2-inkscape.pdf' (pdf, eps, ps)
\begingroup%
  \makeatletter%
  \providecommand\color[2][]{%
    \errmessage{(Inkscape) Color is used for the text in Inkscape, but the package 'color.sty' is not loaded}%
    \renewcommand\color[2][]{}%
  }%
  \providecommand\transparent[1]{%
    \errmessage{(Inkscape) Transparency is used (non-zero) for the text in Inkscape, but the package 'transparent.sty' is not loaded}%
    \renewcommand\transparent[1]{}%
  }%
  \providecommand\rotatebox[2]{#2}%
  \newcommand*\fsize{\dimexpr\f@size pt\relax}%
  \newcommand*\lineheight[1]{\fontsize{\fsize}{#1\fsize}\selectfont}%
  \ifx\svgwidth\undefined%
    \setlength{\unitlength}{377.00787402bp}%
    \ifx\svgscale\undefined%
      \relax%
    \else%
      \setlength{\unitlength}{\unitlength * \real{\svgscale}}%
    \fi%
  \else%
    \setlength{\unitlength}{\svgwidth}%
  \fi%
  \global\let\svgwidth\undefined%
  \global\let\svgscale\undefined%
  \makeatother%
  \begin{picture}(1,0.86466165)%
    \lineheight{1}%
    \setlength\tabcolsep{0pt}%
    \put(0,0){\includegraphics[width=\unitlength,page=1]{figure-pushing-2-inkscape.pdf}}%
    \put(0.3801569,0.47567654){\makebox(0,0)[lt]{\lineheight{1.25}\smash{\begin{tabular}[t]{l}$g$\end{tabular}}}}%
    \put(0.28515054,0.50837733){\makebox(0,0)[lt]{\lineheight{1.25}\smash{\begin{tabular}[t]{l}$s_1^{-1}$\end{tabular}}}}%
    \put(0.12584414,0.54876766){\makebox(0,0)[lt]{\lineheight{1.25}\smash{\begin{tabular}[t]{l}{\small $t$}\end{tabular}}}}%
    \put(0.33961805,0.54546127){\makebox(0,0)[lt]{\lineheight{1.25}\smash{\begin{tabular}[t]{l}$s_1$\end{tabular}}}}%
    \put(0.48974999,0.56105561){\makebox(0,0)[lt]{\lineheight{1.25}\smash{\begin{tabular}[t]{l}$s_4^{-1}$\end{tabular}}}}%
    \put(0.54065595,0.52195809){\makebox(0,0)[lt]{\lineheight{1.25}\smash{\begin{tabular}[t]{l}$s_4$\end{tabular}}}}%
    \put(0.5179043,0.31612422){\makebox(0,0)[lt]{\lineheight{1.25}\smash{\begin{tabular}[t]{l}$s_3$\end{tabular}}}}%
    \put(0.56043571,0.36116556){\makebox(0,0)[lt]{\lineheight{1.25}\smash{\begin{tabular}[t]{l}$s_3^{-1}$\end{tabular}}}}%
    \put(0.20311496,0.42917931){\makebox(0,0)[lt]{\lineheight{1.25}\smash{\begin{tabular}[t]{l}$M_{\widehat{r}_2}$\end{tabular}}}}%
    \put(0.42847701,0.65086697){\makebox(0,0)[lt]{\lineheight{1.25}\smash{\begin{tabular}[t]{l}$M_{\widehat{r}_1}$\end{tabular}}}}%
    \put(0.65155219,0.46533708){\makebox(0,0)[lt]{\lineheight{1.25}\smash{\begin{tabular}[t]{l}$M_{\widehat{r}_4}$\end{tabular}}}}%
    \put(0.46302801,0.22406413){\makebox(0,0)[lt]{\lineheight{1.25}\smash{\begin{tabular}[t]{l}$M_{\widehat{r}_3}$\end{tabular}}}}%
    \put(0,0){\includegraphics[width=\unitlength,page=2]{figure-pushing-2-inkscape.pdf}}%
    \put(0.37175802,0.33360084){\makebox(0,0)[lt]{\lineheight{1.25}\smash{\begin{tabular}[t]{l}$s_2^{-1}$\end{tabular}}}}%
    \put(0.32547195,0.36470634){\makebox(0,0)[lt]{\lineheight{1.25}\smash{\begin{tabular}[t]{l}$s_2$\end{tabular}}}}%
    \put(0.12367297,0.38720583){\makebox(0,0)[lt]{\lineheight{1.25}\smash{\begin{tabular}[t]{l}{\small $t$}\end{tabular}}}}%
    \put(0.19205202,0.30745267){\makebox(0,0)[lt]{\lineheight{1.25}\smash{\begin{tabular}[t]{l}{\small $t$}\end{tabular}}}}%
    \put(0.29740823,0.29127583){\makebox(0,0)[lt]{\lineheight{1.25}\smash{\begin{tabular}[t]{l}{\small $t$}\end{tabular}}}}%
    \put(0.3445699,0.41310877){\makebox(0,0)[lt]{\lineheight{1.25}\smash{\begin{tabular}[t]{l}{\small $t$}\end{tabular}}}}%
    \put(0.25914862,0.57376325){\makebox(0,0)[lt]{\lineheight{1.25}\smash{\begin{tabular}[t]{l}{\small $t$}\end{tabular}}}}%
    \put(0.35228594,0.2427281){\makebox(0,0)[lt]{\lineheight{1.25}\smash{\begin{tabular}[t]{l}{\small $t$}\end{tabular}}}}%
    \put(0.42868449,0.09724363){\makebox(0,0)[lt]{\lineheight{1.25}\smash{\begin{tabular}[t]{l}{\small $t$}\end{tabular}}}}%
    \put(0.62472672,0.18142035){\makebox(0,0)[lt]{\lineheight{1.25}\smash{\begin{tabular}[t]{l}{\small $t$}\end{tabular}}}}%
    \put(0.60607784,0.26858962){\makebox(0,0)[lt]{\lineheight{1.25}\smash{\begin{tabular}[t]{l}{\small $t$}\end{tabular}}}}%
    \put(0.42926717,0.37442854){\makebox(0,0)[lt]{\lineheight{1.25}\smash{\begin{tabular}[t]{l}{\small $t$}\end{tabular}}}}%
    \put(0.68334181,0.31938877){\makebox(0,0)[lt]{\lineheight{1.25}\smash{\begin{tabular}[t]{l}{\small $t$}\end{tabular}}}}%
    \put(0.76562907,0.35342826){\makebox(0,0)[lt]{\lineheight{1.25}\smash{\begin{tabular}[t]{l}{\small $t$}\end{tabular}}}}%
    \put(0.82249273,0.60441389){\makebox(0,0)[lt]{\lineheight{1.25}\smash{\begin{tabular}[t]{l}{\small $t$}\end{tabular}}}}%
    \put(0.66268284,0.59050209){\makebox(0,0)[lt]{\lineheight{1.25}\smash{\begin{tabular}[t]{l}{\small $t$}\end{tabular}}}}%
    \put(0.46806806,0.45190963){\makebox(0,0)[lt]{\lineheight{1.25}\smash{\begin{tabular}[t]{l}{\small $t$}\end{tabular}}}}%
    \put(0.41783449,0.50413605){\makebox(0,0)[lt]{\lineheight{1.25}\smash{\begin{tabular}[t]{l}{\small $t$}\end{tabular}}}}%
    \put(0.58068182,0.64723135){\makebox(0,0)[lt]{\lineheight{1.25}\smash{\begin{tabular}[t]{l}{\small $t$}\end{tabular}}}}%
    \put(0.53110671,0.75435268){\makebox(0,0)[lt]{\lineheight{1.25}\smash{\begin{tabular}[t]{l}{\small $t$}\end{tabular}}}}%
    \put(0.39729622,0.74707793){\makebox(0,0)[lt]{\lineheight{1.25}\smash{\begin{tabular}[t]{l}{\small $t$}\end{tabular}}}}%
    \put(0.31748093,0.63898763){\makebox(0,0)[lt]{\lineheight{1.25}\smash{\begin{tabular}[t]{l}{\small $t$}\end{tabular}}}}%
  \end{picture}%
\endgroup%

%% file: figure-pushing-3-inkscape.pdf_tex
%% Creator: Inkscape 1.3 (0e150ed6c4, 2023-07-21), www.inkscape.org
%% Accompanies image file 'figure-pushing-3-inkscape.pdf' (pdf, eps, ps)
\begingroup%
  \makeatletter%
  \providecommand\color[2][]{%
    \errmessage{(Inkscape) Color is used for the text in Inkscape, but the package 'color.sty' is not loaded}%
    \renewcommand\color[2][]{}%
  }%
  \providecommand\transparent[1]{%
    \errmessage{(Inkscape) Transparency is used (non-zero) for the text in Inkscape, but the package 'transparent.sty' is not loaded}%
    \renewcommand\transparent[1]{}%
  }%
  \providecommand\rotatebox[2]{#2}%
  \newcommand*\fsize{\dimexpr\f@size pt\relax}%
  \newcommand*\lineheight[1]{\fontsize{\fsize}{#1\fsize}\selectfont}%
  \ifx\svgwidth\undefined%
    \setlength{\unitlength}{377.00787402bp}%
    \ifx\svgscale\undefined%
      \relax%
    \else%
      \setlength{\unitlength}{\unitlength * \real{\svgscale}}%
    \fi%
  \else%
    \setlength{\unitlength}{\svgwidth}%
  \fi%
  \global\let\svgwidth\undefined%
  \global\let\svgscale\undefined%
  \makeatother%
  \begin{picture}(1,0.86466165)%
    \lineheight{1}%
    \setlength\tabcolsep{0pt}%
    \put(0,0){\includegraphics[width=\unitlength,page=1]{figure-pushing-3-inkscape.pdf}}%
    \put(0.4387195,0.43042138){\makebox(0,0)[lt]{\lineheight{1.25}\smash{\begin{tabular}[t]{l}$g\cdot t$\end{tabular}}}}%
    \put(0.12584414,0.54876766){\makebox(0,0)[lt]{\lineheight{1.25}\smash{\begin{tabular}[t]{l}{\small $t$}\end{tabular}}}}%
    \put(0.25234882,0.43269601){\makebox(0,0)[lt]{\lineheight{1.25}\smash{\begin{tabular}[t]{l}$M_{\widehat{r}_2}$\end{tabular}}}}%
    \put(0.42707033,0.60163311){\makebox(0,0)[lt]{\lineheight{1.25}\smash{\begin{tabular}[t]{l}$M_{\widehat{r}_1}$\end{tabular}}}}%
    \put(0.59880163,0.45549031){\makebox(0,0)[lt]{\lineheight{1.25}\smash{\begin{tabular}[t]{l}$M_{\widehat{r}_4}$\end{tabular}}}}%
    \put(0.46373135,0.27681468){\makebox(0,0)[lt]{\lineheight{1.25}\smash{\begin{tabular}[t]{l}$M_{\widehat{r}_3}$\end{tabular}}}}%
    \put(0,0){\includegraphics[width=\unitlength,page=2]{figure-pushing-3-inkscape.pdf}}%
    \put(0.12367297,0.38720583){\makebox(0,0)[lt]{\lineheight{1.25}\smash{\begin{tabular}[t]{l}{\small $t$}\end{tabular}}}}%
    \put(0.19205202,0.30745267){\makebox(0,0)[lt]{\lineheight{1.25}\smash{\begin{tabular}[t]{l}{\small $t$}\end{tabular}}}}%
    \put(0.35026912,0.27463698){\makebox(0,0)[lt]{\lineheight{1.25}\smash{\begin{tabular}[t]{l}{\small $t$}\end{tabular}}}}%
    \put(0.29264522,0.60972156){\makebox(0,0)[lt]{\lineheight{1.25}\smash{\begin{tabular}[t]{l}{\small $t$}\end{tabular}}}}%
    \put(0.42868449,0.09724363){\makebox(0,0)[lt]{\lineheight{1.25}\smash{\begin{tabular}[t]{l}{\small $t$}\end{tabular}}}}%
    \put(0.62472672,0.18142035){\makebox(0,0)[lt]{\lineheight{1.25}\smash{\begin{tabular}[t]{l}{\small $t$}\end{tabular}}}}%
    \put(0.65034083,0.29544584){\makebox(0,0)[lt]{\lineheight{1.25}\smash{\begin{tabular}[t]{l}{\small $t$}\end{tabular}}}}%
    \put(0.76562907,0.35342826){\makebox(0,0)[lt]{\lineheight{1.25}\smash{\begin{tabular}[t]{l}{\small $t$}\end{tabular}}}}%
    \put(0.82249273,0.60441389){\makebox(0,0)[lt]{\lineheight{1.25}\smash{\begin{tabular}[t]{l}{\small $t$}\end{tabular}}}}%
    \put(0.60219551,0.61019562){\makebox(0,0)[lt]{\lineheight{1.25}\smash{\begin{tabular}[t]{l}{\small $t$}\end{tabular}}}}%
    \put(0.53110671,0.75435268){\makebox(0,0)[lt]{\lineheight{1.25}\smash{\begin{tabular}[t]{l}{\small $t$}\end{tabular}}}}%
    \put(0.39729622,0.74707793){\makebox(0,0)[lt]{\lineheight{1.25}\smash{\begin{tabular}[t]{l}{\small $t$}\end{tabular}}}}%
    \put(0,0){\includegraphics[width=\unitlength,page=3]{figure-pushing-3-inkscape.pdf}}%
  \end{picture}%
\endgroup%

%% file: figure-replacing-old-cell-inkscape.pdf_tex
%% Creator: Inkscape 1.3 (0e150ed6c4, 2023-07-21), www.inkscape.org
%% Accompanies image file 'figure-replacing-old-cell-inkscape.pdf' (pdf, eps, ps)
\begingroup%
  \makeatletter%
  \providecommand\color[2][]{%
    \errmessage{(Inkscape) Color is used for the text in Inkscape, but the package 'color.sty' is not loaded}%
    \renewcommand\color[2][]{}%
  }%
  \providecommand\transparent[1]{%
    \errmessage{(Inkscape) Transparency is used (non-zero) for the text in Inkscape, but the package 'transparent.sty' is not loaded}%
    \renewcommand\transparent[1]{}%
  }%
  \providecommand\rotatebox[2]{#2}%
  \newcommand*\fsize{\dimexpr\f@size pt\relax}%
  \newcommand*\lineheight[1]{\fontsize{\fsize}{#1\fsize}\selectfont}%
  \ifx\svgwidth\undefined%
    \setlength{\unitlength}{496.06299213bp}%
    \ifx\svgscale\undefined%
      \relax%
    \else%
      \setlength{\unitlength}{\unitlength * \real{\svgscale}}%
    \fi%
  \else%
    \setlength{\unitlength}{\svgwidth}%
  \fi%
  \global\let\svgwidth\undefined%
  \global\let\svgscale\undefined%
  \makeatother%
  \begin{picture}(1,0.37142857)%
    \lineheight{1}%
    \setlength\tabcolsep{0pt}%
    \put(0,0){\includegraphics[width=\unitlength,page=1]{figure-replacing-old-cell-inkscape.pdf}}%
    \put(0.67885171,0.2942582){\makebox(0,0)[lt]{\lineheight{1.25}\smash{\begin{tabular}[t]{l}$r_2$\end{tabular}}}}%
    \put(0.56445251,0.19287864){\makebox(0,0)[lt]{\lineheight{1.25}\smash{\begin{tabular}[t]{l}$r_3$\end{tabular}}}}%
    \put(0.84113107,0.22160483){\makebox(0,0)[lt]{\lineheight{1.25}\smash{\begin{tabular}[t]{l}$M_{\widehat{r}_1}$\end{tabular}}}}%
    \put(0.68238107,0.10367626){\makebox(0,0)[lt]{\lineheight{1.25}\smash{\begin{tabular}[t]{l}$r_4$\end{tabular}}}}%
    \put(0.12901284,0.2947017){\makebox(0,0)[lt]{\lineheight{1.25}\smash{\begin{tabular}[t]{l}$r_2$\end{tabular}}}}%
    \put(0.01461365,0.19332213){\makebox(0,0)[lt]{\lineheight{1.25}\smash{\begin{tabular}[t]{l}$r_3$\end{tabular}}}}%
    \put(0.2912922,0.22204833){\makebox(0,0)[lt]{\lineheight{1.25}\smash{\begin{tabular}[t]{l}$r_1$\end{tabular}}}}%
    \put(0.1325422,0.10411974){\makebox(0,0)[lt]{\lineheight{1.25}\smash{\begin{tabular}[t]{l}$r_4$\end{tabular}}}}%
    \put(0.11719994,0.21378642){\makebox(0,0)[lt]{\lineheight{1.25}\smash{\begin{tabular}[t]{l}$v$\end{tabular}}}}%
    \put(0.65984881,0.20799768){\makebox(0,0)[lt]{\lineheight{1.25}\smash{\begin{tabular}[t]{l}$v$\end{tabular}}}}%
    \put(0,0){\includegraphics[width=\unitlength,page=2]{figure-replacing-old-cell-inkscape.pdf}}%
    \put(0.73772923,0.19810474){\makebox(0,0)[lt]{\lineheight{1.25}\smash{\begin{tabular}[t]{l}$t_1$\end{tabular}}}}%
    \put(0.86751256,0.09570033){\makebox(0,0)[lt]{\lineheight{1.25}\smash{\begin{tabular}[t]{l}$t_1$\end{tabular}}}}%
    \put(0.92321638,0.17628115){\makebox(0,0)[lt]{\lineheight{1.25}\smash{\begin{tabular}[t]{l}$t_1$\end{tabular}}}}%
    \put(0.92002398,0.26380307){\makebox(0,0)[lt]{\lineheight{1.25}\smash{\begin{tabular}[t]{l}$t_1$\end{tabular}}}}%
    \put(0.89748788,0.31562052){\makebox(0,0)[lt]{\lineheight{1.25}\smash{\begin{tabular}[t]{l}$t_1$\end{tabular}}}}%
    \put(0.82492022,0.32220845){\makebox(0,0)[lt]{\lineheight{1.25}\smash{\begin{tabular}[t]{l}$t_1$\end{tabular}}}}%
    \put(0.77748125,0.14899632){\makebox(0,0)[lt]{\lineheight{1.25}\smash{\begin{tabular}[t]{l}$s_{x_1,t_1}$\end{tabular}}}}%
    \put(0.75548192,0.11378379){\makebox(0,0)[lt]{\lineheight{1.25}\smash{\begin{tabular}[t]{l}$x_1$\end{tabular}}}}%
    \put(0.82246286,0.17747096){\makebox(0,0)[lt]{\lineheight{1.25}\smash{\begin{tabular}[t]{l}$w_{x_1,t_1}$\end{tabular}}}}%
    \put(0.95782052,0.30498777){\makebox(0,0)[lt]{\lineheight{1.25}\smash{\begin{tabular}[t]{l}$v_j$\end{tabular}}}}%
    \put(0.41052699,0.30985421){\makebox(0,0)[lt]{\lineheight{1.25}\smash{\begin{tabular}[t]{l}$v_j$\end{tabular}}}}%
  \end{picture}%
\endgroup%

%% file: figure-replacing-new-cell-inkscape.pdf_tex
%% Creator: Inkscape 1.3 (0e150ed6c4, 2023-07-21), www.inkscape.org
%% Accompanies image file 'figure-replacing-new-cell-inkscape.pdf' (pdf, eps, ps)
\begingroup%
  \makeatletter%
  \providecommand\color[2][]{%
    \errmessage{(Inkscape) Color is used for the text in Inkscape, but the package 'color.sty' is not loaded}%
    \renewcommand\color[2][]{}%
  }%
  \providecommand\transparent[1]{%
    \errmessage{(Inkscape) Transparency is used (non-zero) for the text in Inkscape, but the package 'transparent.sty' is not loaded}%
    \renewcommand\transparent[1]{}%
  }%
  \providecommand\rotatebox[2]{#2}%
  \newcommand*\fsize{\dimexpr\f@size pt\relax}%
  \newcommand*\lineheight[1]{\fontsize{\fsize}{#1\fsize}\selectfont}%
  \ifx\svgwidth\undefined%
    \setlength{\unitlength}{439.37007874bp}%
    \ifx\svgscale\undefined%
      \relax%
    \else%
      \setlength{\unitlength}{\unitlength * \real{\svgscale}}%
    \fi%
  \else%
    \setlength{\unitlength}{\svgwidth}%
  \fi%
  \global\let\svgwidth\undefined%
  \global\let\svgscale\undefined%
  \makeatother%
  \begin{picture}(1,0.41935484)%
    \lineheight{1}%
    \setlength\tabcolsep{0pt}%
    \put(0,0){\includegraphics[width=\unitlength,page=1]{figure-replacing-new-cell-inkscape.pdf}}%
    \put(0.10530103,0.34409403){\makebox(0,0)[lt]{\lineheight{1.25}\smash{\begin{tabular}[t]{l}$r_2$\end{tabular}}}}%
    \put(0.02555938,0.24857708){\makebox(0,0)[lt]{\lineheight{1.25}\smash{\begin{tabular}[t]{l}$r_3$\end{tabular}}}}%
    \put(0.23695945,0.28661067){\makebox(0,0)[lt]{\lineheight{1.25}\smash{\begin{tabular}[t]{l}$s_{x_1,t_1}$\end{tabular}}}}%
    \put(0.12098155,0.1541244){\makebox(0,0)[lt]{\lineheight{1.25}\smash{\begin{tabular}[t]{l}$r_4$\end{tabular}}}}%
    \put(0.08125648,0.28799011){\makebox(0,0)[lt]{\lineheight{1.25}\smash{\begin{tabular}[t]{l}$v$\end{tabular}}}}%
    \put(0,0){\includegraphics[width=\unitlength,page=2]{figure-replacing-new-cell-inkscape.pdf}}%
    \put(0.16736477,0.31590496){\makebox(0,0)[lt]{\lineheight{1.25}\smash{\begin{tabular}[t]{l}$t_1$\end{tabular}}}}%
    \put(0.3090318,0.19202866){\makebox(0,0)[lt]{\lineheight{1.25}\smash{\begin{tabular}[t]{l}$t_1$\end{tabular}}}}%
    \put(0.19157997,0.2398001){\makebox(0,0)[lt]{\lineheight{1.25}\smash{\begin{tabular}[t]{l}$x_1$\end{tabular}}}}%
    \put(0.24758786,0.16369526){\makebox(0,0)[lt]{\lineheight{1.25}\smash{\begin{tabular}[t]{l}$v_j$\end{tabular}}}}%
    \put(0,0){\includegraphics[width=\unitlength,page=3]{figure-replacing-new-cell-inkscape.pdf}}%
    \put(0.66168641,0.33166502){\makebox(0,0)[lt]{\lineheight{1.25}\smash{\begin{tabular}[t]{l}$r_2$\end{tabular}}}}%
    \put(0.58194467,0.23614807){\makebox(0,0)[lt]{\lineheight{1.25}\smash{\begin{tabular}[t]{l}$r_3$\end{tabular}}}}%
    \put(0.75864824,0.27014254){\makebox(0,0)[lt]{\lineheight{1.25}\smash{\begin{tabular}[t]{l}$M_{\widehat{s}_{x_1,t_1}}$\end{tabular}}}}%
    \put(0.69282118,0.16659374){\makebox(0,0)[lt]{\lineheight{1.25}\smash{\begin{tabular}[t]{l}$M_{\widehat{r}_4}$\end{tabular}}}}%
    \put(0.63764176,0.27556111){\makebox(0,0)[lt]{\lineheight{1.25}\smash{\begin{tabular}[t]{l}$v$\end{tabular}}}}%
    \put(0,0){\includegraphics[width=\unitlength,page=4]{figure-replacing-new-cell-inkscape.pdf}}%
    \put(0.69253319,0.27738419){\makebox(0,0)[lt]{\lineheight{1.25}\smash{\begin{tabular}[t]{l}$t_2$\end{tabular}}}}%
    \put(0.61772466,0.16734494){\makebox(0,0)[lt]{\lineheight{1.25}\smash{\begin{tabular}[t]{l}$t_2$\end{tabular}}}}%
    \put(0.65350357,0.07694589){\makebox(0,0)[lt]{\lineheight{1.25}\smash{\begin{tabular}[t]{l}$t_2$\end{tabular}}}}%
    \put(0.74173737,0.05691954){\makebox(0,0)[lt]{\lineheight{1.25}\smash{\begin{tabular}[t]{l}$t_2$\end{tabular}}}}%
    \put(0.7815507,0.20010031){\makebox(0,0)[lt]{\lineheight{1.25}\smash{\begin{tabular}[t]{l}$t_2$\end{tabular}}}}%
    \put(0.89792067,0.28232257){\makebox(0,0)[lt]{\lineheight{1.25}\smash{\begin{tabular}[t]{l}$t_2$\end{tabular}}}}%
    \put(0.81465946,0.34673484){\makebox(0,0)[lt]{\lineheight{1.25}\smash{\begin{tabular}[t]{l}$t_2$\end{tabular}}}}%
    \put(0.89038115,0.34688315){\makebox(0,0)[lt]{\lineheight{1.25}\smash{\begin{tabular}[t]{l}$w_{x_1,t_1}$\end{tabular}}}}%
    \put(0.3406964,0.3594463){\makebox(0,0)[lt]{\lineheight{1.25}\smash{\begin{tabular}[t]{l}$w_{x_1,t_1}$\end{tabular}}}}%
    \put(0.80397314,0.15126623){\makebox(0,0)[lt]{\lineheight{1.25}\smash{\begin{tabular}[t]{l}$v_j$\end{tabular}}}}%
    \put(0,0){\includegraphics[width=\unitlength,page=5]{figure-replacing-new-cell-inkscape.pdf}}%
  \end{picture}%
\endgroup%